\documentclass[12pt,a4paper]{amsart}
\usepackage{graphics}
\usepackage{epsfig}
\usepackage{graphicx}
\usepackage{amscd}
\usepackage[all]{xy}
\theoremstyle{plain}
\usepackage{amssymb}
\usepackage{enumerate}

\advance\hoffset-20mm \advance\textwidth41mm

\newtheorem{theorem}{Theorem}
\newtheorem{lemma}{Lemma}
\newtheorem*{theo*}{Theorem}
\newtheorem{proposition}{Proposition}
\newtheorem{corollary}{Corollary}
\theoremstyle{definition}
\newtheorem{definition}{Definition}
\newtheorem*{definition*}{Definition}
\newtheorem{example}{Example}
\newtheorem{remark}{Remark}
\newtheorem{problem}{Problem}

\def\PP{{\mathbb P}}
\def\FF{{\mathbb F}}
\def\QQ{{\mathbb Q}}

\def\KK{{\mathbb K}}
\def\ZZ{{\mathbb Z}}

\def\GG{{\mathbb G}}
\def\CC{{\mathbb C}}
\def\NN{{\mathbb N}}

\def\fu{{\mathfrak u}}

\def\div{\mathop{\rm div}}

\def\dg{\mathop{\rm deg}}

\def\Spec{\mathop{\rm Spec}}

\def\PSL{\mathop{\rm PSL}}
\def\Gr{\mathop{\rm Grass}}
\def\Aut{\mathop{\rm Aut}}

\def\WDiv{\mathop{\rm WDiv}}

\def\Cl{\mathop{\rm Cl}}

\newcommand{\CO}{{\mathcal{O}}}

\DeclareMathOperator{\rk}{rk}

\begin{document}
\sloppy
\title[Additive actions on toric varieties]
{Additive actions on toric varieties}
\author{Ivan Arzhantsev}
\address{National Research University Higher School of Economics, Faculty
of Computer Science, Kochnovskiy Proezd 3, Moscow, 125319 Russia}
\email{arjantsev@hse.ru}
\author{Elena Romaskevich}
\address{Yandex, ulica L'va Tolstogo 16, Moscow, 119034 Russia}
\email{lena.apq@gmail.com}
\date{\today}
\begin{abstract}
By an additive action on an algebraic variety $X$ of dimension $n$ we mean a regular action $\GG_a^n\times X\to X$ with an open orbit of the commutative unipotent group $\GG_a^n$. We prove that if a complete toric variety $X$ admits an additive action, then it admits an additive action normalized by the acting torus. Normalized additive actions on a toric variety $X$ are in bijection with complete collections of Demazure roots of the fan $\Sigma_X$. Moreover, any two normalized additive actions on $X$ are isomorphic.
\end{abstract}
\maketitle
\section*{Introduction}

Let $X$ be an irreducible algebraic variety of dimension $n$ over an algebraically closed field $\KK$ of characteristic zero and let $\GG_a=(\KK,+)$ be the additive group of the field. Consider the commutative unipotent group $\GG_a^n=\GG_a\times\ldots\times\GG_a$ ($n$~times). By an additive action on $X$ we mean a regular action $\GG_a^n\times X\to X$ with an open orbit. Equivalently, one may consider algebraic varieties with an additive action as equivariant embeddings of the vector group $(\KK^n,+)$.

A systematic study of additive actions was initiated by Hassett and Tschinkel~\cite{HT}. They established a remarkable correspondence between additive actions on the projective space $\PP^n$  and local $(n+1)$-dimensional commutative associative algebras with unit. This correspondence has allowed to classify additive actions on $\PP^n$ for small $n$. The same technique was used by Sharoiko~\cite{Sh} to prove that an additive action on a non-degenerate projective quadric is unique. Further modification of Hassett-Tschinkel correspondence led to characterization of additive actions on arbitrary projective hypersurfaces, in particular, on degenerate projective quadrics~\cite{AS},~\cite{AP}.

The study of additive actions was originally motivated by problems of arithmetic geometry. Chambert-Loir and Tschinkel~\cite{CLT1} gave asymptotic formulas for the number of rational points of bounded height on smooth projective equivariant compactifications of the vector group.
More generally, asymptotic formulas for the number of rational points of bounded height on quasi-projective equivariant embeddings of the vector group are obtained in~\cite{CLT2}.

In~\cite{Ar} all generalized flag varieties $G/P$ admitting an additive action are found. Here $G$ is a semisimple linear algebraic group and $P$ is a parabolic subgroup. It turns out that if $G/P$ admits an additive action then the parabolic subgroup $P$ is maximal.

Feigin~\cite{Fe} proposed a construction based on the PBW-filtration to degenerate an arbitrary generalized flag variety $G/P$ to a variety with an additive action. Recently Fu-Hwang~\cite{FH} and Devyatov~\cite{Dev} have proved that if $G/P$ is not isomorphic to the projective space, then up to isomorphism there is at most one additive action on $G/P$. Classification of additive actions on singular del Pezzo surfaces is obtained by Derenthal and Loughran~\cite{DL}.

\smallskip

The problem of classification of additive actions on toric varieties is raised in~\cite[Section~6]{AS}. Some instructive examples of such actions are given in~\cite[Proposition~5.5]{HT}. It is natural to divide the problem into two parts. The first one deals with additive actions on a toric variety $X$ of dimension $n$ normalized by the acting torus $T$. In this case an additive action splits into $n$ pairwise commuting $\GG_a$-actions on $X$ normalized by $T$. It is proved in~\cite{De} that $\GG_a$-actions on a toric variety $X$ normalized by $T$ are in bijection with some vectors defined in terms of the fan $\Sigma_X$ associated with $X$. We call such vectors Demazure roots of a fan. Cox~\cite{Cox} observed that normalized $\GG_a$-actions on a toric
variety can be interpreted as certain $\GG_a$-subgroups of automorphisms of the Cox ring $R(X)$ of the variety $X$. In turn, such subgroups correspond to homogeneous locally nilpotent derivations of the Cox ring. In these terms the Demazure root is nothing but the degree of the derivation.

After presenting some background on $\GG_a$-actions and Demazure roots (Section~\ref{sec1}) and on toric varieties and Cox rings (Section~\ref{sec2}), we prove in Theorem~\ref{cc} that additive actions on a toric variety $X$ normalized by the acting torus $T$
are in bijection with complete collections of Demazure roots of the fan $\Sigma_X$; see Definition~\ref{deff} for a precise definition of a complete collection of Demazure roots.
In Theorem~\ref{unique} we show that any two normalized additive actions on $X$ are isomorphic.

The second part of the problem concerns non-normalized additive actions. Theorem~\ref{3con} states that if a complete toric variety admits an additive action, then it admits an additive action normalized by the acting torus.

It is well known that a toric variety is projective if and only if its fan is a normal fan of a convex polytope. In Section~\ref{sec5} we characterize polytopes corresponding to projective toric varieties with an additive action.

In the last section we give explicit examples of additive actions on toric varieties in terms of their Cox rings and formulate several open problems.

\smallskip

The authors are grateful to Evgeny Feigin for useful discussions and encouragement.


\section{$\GG_a$-actions and Demazure roots}
\label{sec1}

Consider an irreducible affine variety $X$ with an effective action of an
algebraic torus $T$. Let $M$ be the character lattice of $T$ and $N$ be the lattice of one-parameter subgroups of $T$. Let $A=\KK[X]$ be the algebra of regular functions on $X$. It is well known that there is a bijection between faithful $T$-actions on $X$ and effective $M$-gradings on~$A$. In fact, the algebra $A$ is graded by a semigroup of lattice points in a convex polyhedral cone $\omega\subseteq M_{\QQ}=M\otimes_{\ZZ}\QQ$. We have
$$
A=\bigoplus_{m\in \omega_{ M}} A_m\chi^m,
$$
where $\omega_{M}=\omega\cap M$ and $\chi^m$ is the character of the torus $T$ corresponding to a point $m$.

A derivation $\partial$ of an algebra $A$ is said to be {\itshape locally nilpotent} (LND) if for every $f\in A$ there exists $k\in\NN$ such that $\partial^k(f)=0$. For any LND $\partial$ on $A$ the map ${\varphi_{\partial}:\GG_a\times A\rightarrow A}$, ${\varphi_{\partial}(s,f)=\exp(s\partial)(f)}$, defines a
structure of a rational $\GG_a$-algebra on $A$. This induces a regular action $\GG_a\times X\to X$, where $X=\Spec A$. In fact, any regular $\GG_a$-action on $X$ arises this way, see \cite[Section~1.5]{Fr}. A derivation $\partial$ on $A$ is said to be {\itshape homogeneous} if it respects the $M$-grading. If ${f,h\in A\backslash \ker\partial}$ are homogeneous, then
${\partial(fh)=f\partial(h)+\partial(f)h}$ is homogeneous too and
${\dg\partial(f)-\dg f=\dg\partial(h)-\dg h}$. So any homogeneous
derivation $\partial$ has a well defined {\itshape degree} given
as $\dg\partial=\dg\partial(f)-\dg f$ for any homogeneous $f\in A\backslash \ker\partial$.
It is easy to see that an LND on $A$ is homogeneous if and only if the corresponding
$\GG_a$-action is normalized by the torus~$T$ in the automorphism group~$\Aut (X)$, cf. \cite[Section~3.7]{Fr}.

Let $X$ be an affine toric variety, i.~e. a normal affine variety with an action of a torus $T$ with an open orbit. In this case
$$
A=\bigoplus_{m\in \omega_{M}}\KK\chi^m=\KK[\omega_{M}]
$$
is the semigroup algebra. Recall that for a given cone $\omega\subset M_{\QQ}$, its {\itshape dual cone} is defined by
$$
\sigma=\{p\in N_{\mathbb{Q}}\,|\,\langle p,v\rangle\geqslant0\,\,\,\forall v\in\omega\},
$$
where $\langle,\rangle$ is the pairing between dual spaces $N_{\mathbb{Q}}$ and $M_{\mathbb{Q}}$. Let $\sigma(1)$ be the set of rays of a cone $\sigma$ and $p_{\rho}$ be the primitive lattice vector on a ray $\rho$. For $\rho\in\sigma(1)$ we set
$$
\mathfrak{R}_{\rho}:=\{e\in M\,|\, \langle p_{\rho},e\rangle=-1
\,\,\mbox{and}\,\, \langle p_{\rho'},e\rangle\geqslant0
\,\,\,\,\forall\,\rho'\in \sigma(1), \,\rho'\ne\rho\}.
$$
One easily checks that the set $\mathfrak{R}_{\rho}$ is infinite for each $\rho\in\sigma(1)$. The elements of the set $\mathfrak{R}:=\bigsqcup\limits_{\rho}\mathfrak{R}_{\rho}$ are called the {\itshape Demazure roots} of the cone $\sigma$. Let $e\in\mathfrak{R}_{\rho}$. Then $\rho$ is called the {\itshape distinguished ray} of the root $e$. One can define the homogeneous LND on the algebra $A$ by the rule
$$
\partial_e(\chi^m)=\langle p_{\rho},m\rangle\chi^{m+e}.
$$
In fact, every homogeneous LND on~$A$ has a form $\alpha\partial_e$ for some $\alpha\in
\KK,\, e\in \mathfrak{R}$, see \cite[Theorem~2.7]{L1}. In other words, $\GG_a$-actions on $X$
normalized by the acting torus are in bijection with Demazure roots of the cone $\sigma$.

\begin{example} \label{ex1}
Consider $X=\KK^n$ with the standard action of the torus $(\KK^{\times})^n$.
It is a toric variety with the cone $\sigma=\QQ^n_{\geqslant0}$ having rays
$\rho_i=\langle p_i\rangle_{\QQ_{\geqslant0}}$ with
$p_1=(1,0,\ldots,0),\ldots,p_n=(0,\ldots,0,1)$.
The dual cone $\omega$ is $\QQ^n_{\geqslant0}$ as well. In this case we have
$$
\mathfrak{R}_{\rho_i}=\{(c_1,\ldots,c_{i-1},-1,c_{i+1},\ldots,c_n)\,|\,c_j\in\ZZ_{\geqslant0}\},
$$
\vspace{0.05cm}
\begin{center}
\begin{picture}(100,75)
\multiput(50,15)(15,0){5}{\circle*{3}}
\multiput(35,30)(0,15){4}{\circle*{3}}
\put(20,30){\vector(1,0){100}} \put(50,5){\vector(0,1){80}}
\put(17,70){$\mathfrak{R}_{\rho_1}$} \put(115,7){$\mathfrak{R}_{\rho_2}$}
\put(100,70){$M_{\mathbb{Q}}=\mathbb{Q}^2$} \linethickness{0.5mm}
\put(50,30){\line(1,0){65}} \put(50,30){\line(0,1){50}}
\end{picture}
\end{center}
where $c_j=\langle p_j,e\rangle$. Denote
$x_1=\chi^{(1,0,\ldots,0)},\ldots,x_n=\chi^{(0,\ldots,0,1)}$. Then
$\KK[X]=\KK[x_1,\ldots,x_n]$. It is easy to see that the homogeneous LND corresponding to the root
$e=(c_1,\ldots,c_n)\in \mathfrak{R}_{\rho_i}$ is
$$
\partial_e=x_1^{c_1}\ldots x_{i-1}^{c_{i-1}} x_{i+1}^{c_{i+1}}\ldots x_n^{c_n}\frac{\partial}{\partial x_i}. \eqno (*)
$$

This LND gives rise to the $\GG_a$-action
$$
x_i\mapsto x_i+sx_1^{c_1}\ldots x_{i-1}^{c_{i-1}} x_{i+1}^{c_{i+1}}\ldots x_n^{c_n}, \quad
x_j\mapsto x_j, \quad j\ne i, \quad s\in\GG_a.
$$
\end{example}


\section{Toric varieties and Cox rings}
\label{sec2}

We keep notation of Section~\ref{sec1}. Let $X$ be a toric variety of dimension $n$ with an
acting torus $T$ and $\Sigma$ be the corresponding fan of convex polyhedral cones in $N_{\QQ}$; see~\cite{Fu} or \cite{CLS} for details.

Let $\Sigma(1)$ be the set of rays of the fan $\Sigma$ and $p_{\rho}$ be the primitive lattice vector on a ray $\rho$. For $\rho\in\Sigma(1)$ we consider the set $\mathfrak{R}_{\rho}$ of all vectors $e\in M$ such that
\begin{enumerate}
\item[(1)]
$\langle p_{\rho},e\rangle=-1\,\,\mbox{and}\,\, \langle p_{\rho'},e\rangle\geqslant0
\,\,\,\,\forall\,\rho'\in \Sigma(1), \,\rho'\ne\rho$;
\smallskip
\item[(2)]
if $\sigma$ is a cone of $\Sigma$ and $\langle v,e\rangle=0$ for all $v\in\sigma$, then the cone generated by $\sigma$ and $\rho$ is in $\Sigma$ as well.
\end{enumerate}

Note that~(1) implies~(2) if $\Sigma$ is a fan with convex support. This is the case if $X$ is affine or complete.

The elements of the set $\mathfrak{R}:=\bigsqcup\limits_{\rho}\mathfrak{R}_{\rho}$ are called the {\itshape Demazure roots} of the fan $\Sigma$, cf.~\cite[D\'efinition~4]{De} and \cite[Section~3.4]{Oda}.
Again elements $e\in\mathfrak{R}$ are in bijection with $\GG_a$-actions on $X$ normalized by the acting torus, see~\cite[Th\'eoreme~3]{De} and \cite[Proposition~3.14]{Oda}. If $X$ is affine, the
$\GG_a$-action given by a Demazure root~$e$ coincides with the action corresponding to the locally nilpotent derivation $\partial_e$ of the algebra $\KK[X]$ as defined in Section~\ref{sec1}.
Let us denote by $H_e$ the image in $\Aut(X)$ of the group $\GG_a$ under this action. Thus $H_e$
is a one-parameter unipotent subgroup normalized by $T$ in $\Aut(X)$.

We recall basic facts from toric geometry. There is a bijection between cones $\sigma\in\Sigma$ and $T$-orbits $\CO_{\sigma}$ on $X$ such that $\sigma_1\subseteq\sigma_2$ if and only if $\CO_{\sigma_2}\subseteq\overline{\CO_{\sigma_1}}$. Here $\dim\CO_{\sigma}=n-\dim\langle\sigma\rangle$. Moreover, each cone $\sigma\in\Sigma$ defines an open affine $T$-invariant subset $U_{\sigma}$ on $X$ such that $\CO_{\sigma}$ is a unique closed $T$-orbit on $U_{\sigma}$ and $\sigma_1\subseteq\sigma_2$ if and only if $U_{\sigma_1}\subseteq U_{\sigma_2}$.

Let $\rho_e$ be the distinguished ray corresponding to a root $e$, $p_e$ be the primitive lattice vector on $\rho_e$, and $R_e$ be the one-parameter subgroup of $T$ corresponding to $p_e$. Denote by $X^{H_e}$ the set of $H_e$-fixed points on $X$.

\smallskip

Our aim is to describe the action of $H_e$ on $X$.

\begin{proposition} \label{phe}
For every point $x\in X\setminus X^{H_e}$ the orbit $H_ex$ meets exactly two $T$-orbits $\CO_1$ and $\CO_2$ on $X$, where $\dim\CO_1=\dim\CO_2+1$. The intersection $\CO_2\cap H_ex$ consists of a single point, while
$$
\CO_1\cap H_ex=R_ey \quad \text{for any} \quad y\in\CO_1\cap H_ex.
$$
\end{proposition}

\begin{proof}
It follows from the proof of \cite[Proposition~3.14]{Oda} that the affine charts $U_{\sigma}$, where $\sigma\in\Sigma$ is a cone containing $\rho_e$, are $H_e$-invariant, and the complement of their union is contained in $X^{H_e}$. This reduces the proof to the case $X$ is affine. Then the assertion is proved in \cite[Proposition~2.1]{AKZ}.
\end{proof}

A pair of $T$-orbits $(\CO_1,\CO_2)$ on $X$ is said to be {\itshape $H_e$-connected} if
$H_ex\subseteq \CO_1\cup\CO_2$ for some $x\in X\setminus X^{H_e}$. By Proposition~\ref{phe}, $\CO_2\subseteq\overline{\CO_1}$ for such a pair (up to permutation) and $\dim\CO_1=\dim\CO_2+1$. Since the torus normalizes the subgroup $H_e$, any point of
$\CO_1\cup\CO_2$ can actually serve as a point $x$.

\begin{lemma} \label{lpe}
A pair of $T$-orbits $(\CO_{\sigma_1},\CO_{\sigma_2})$ is $H_e$-connected if and only if
$e|_{\sigma_2}\le 0$ and $\sigma_1$ is a facet of $\sigma_2$ given by the equation
$\langle v,e\rangle=0$.
\end{lemma}

\begin{proof} The proof again reduces to the affine case, where the assertion is \cite[Lemma~2.2]{AKZ}.
\end{proof}

\smallskip

Now we recall basic ingredients of the Cox construction, see~\cite[Chapter~1]{ADHL} for more details. Let $X$ be a normal variety with free finitely generated divisor class group $\Cl(X)$ and only constant invertible regular functions. Denote by $\WDiv(X)$ the group of Weil divisors on $X$ and fix a subgroup $K \subseteq \WDiv(X)$ which maps onto $\Cl(X)$ isomorphically. The {\itshape Cox ring} of the variety $X$ is defined as
$$
R(X)=\bigoplus_{D\in K} H^0(X,D),
$$
where $H^0(X,D)=\{f\in\KK(X)^{\times}\mid \div(f)+D\geqslant0\}\cup\{0\}$ and multiplication on homogeneous components coincides with multiplication in the field of rational functions $\KK(X)$ and extends to $R(X)$ by linearity. It is easy to see that up to isomorphism the graded ring
$R(X)$ does not depend on the choice of the subgroup $K$.

\begin{example}
It is proved in \cite{Cox} that if $X$ is toric, then $R(X)$ is a polynomial algebra
$\KK[x_1,\ldots,x_m]$, where the variables $x_i$ correspond to $T$-invariant prime divisors
$D_i$ on $X$ or, equivalently, to the rays $\rho_i$ of the fan $\Sigma_X$. The $\Cl(X)$-grading on $R(X)$ is given by $\dg(x_i)=[D_i]$.
\end{example}

Suppose that the Cox ring $R(X)$ is finitely generated. Then ${\overline{X}:=\Spec R(X)}$ is a normal affine variety with an action of the torus $H_X := \Spec\KK[\Cl(X)]$. There is an open $H_X$-invariant subset $\widehat{X}\subseteq \overline{X}$ such that the complement $\overline{X}\backslash\widehat{X}$ is of codimension at least two in $\overline{X}$,
there exists a good quotient $p_X\colon\widehat{X}\rightarrow\widehat{X}/\!/H_{X}$, and the quotient space $\widehat{X}/\!/H_{X}$ is isomorphic to $X$, see \cite[Construction~1.6.3.1]{ADHL}. So we have the following diagram
$$
\begin{CD}
\widehat{X} @>{i}>> \overline{X}=\Spec R(X)\\
@VV{/\!/H_{X}}V  \\
X
\end{CD}
$$
If $X$ is toric, then $\overline{X}$ is isomorphic to $\KK^m$, and $\overline{X}\setminus\widehat{X}$ is a union of some coordinate planes in $\KK^m$ of codimension at least two~\cite{Cox}.

\smallskip

By~\cite[Theorem~4.2.3.2]{ADHL}, any $\GG_a$-action on a variety $X$ can be lifted to a $\GG_a$-action on the variety $\overline{X}$ commuting with the action of the torus $H_X$.

If $X$ is toric and a $\GG_a$-action is normalized by the acting torus $T$, then the lifted
$\GG_a$-action on $\KK^m$ is normalized by the diagonal torus $(\KK^{\times})^m$. Conversely,
any $\GG_a$-action on $\KK^m$ normalized by the torus $(\KK^{\times})^m$ and commuting
with the subtorus $H_X$ induces a $\GG_a$-action on $X$. This shows that $\GG_a$-actions on $X$
normalized by $T$ are in bijection with locally nilpotent derivations of the Cox ring
$\KK[x_1,\ldots,x_m]$ that are homogeneous with respect to the standard grading by the lattice
$\ZZ^m$ and have degree zero with respect to the $\Cl(X)$-grading.


\section{Normalized additive actions}
\label{sec3}

Let $X$ be an irreducible algebraic variety of dimension $n$. Consider a commutative unipotent
algebraic group $\GG_a^n=\GG_a\times\ldots\times\GG_a$ ($n$~times).

\begin{definition}
An {\it additive action} on a variety $X$ is a regular action $\GG_a^n\times X\to X$ with an open orbit.
\end{definition}

Let $X$ be a normal variety admitting an additive action. Then $X$ contains an open $\GG_a^n$-orbit $W$ isomorphic to the affine space $\KK^n$. By~\cite[Lemma~1]{APS}, any invertible function on $X$ is constant and the divisor class group $\Cl(X)$ is freely generated by classes $[D_1],\ldots,[D_l]$ of the prime divisors such that $X\setminus W=D_1\cup\ldots\cup D_l$.
In particular, the Cox ring $R(X)$ introduced in Section~\ref{sec2} is well defined for such a variety $X$.

Now we assume that $X$ is toric and an additive action $\GG_a^n\times X\to X$ is normalized by the acting torus $T$. Then the group $\GG_a^n$ is a direct product of $n$ subgroups $\GG_a$ each normalized by~$T$. They correspond to pairwise commuting homogeneous locally nilpotent derivations on the Cox ring $\KK[x_1,\ldots,x_m]$ having degree zero with respect to the $\Cl(X)$-grading.
In turn, such derivations up to scalar are in bijection with Demazure roots of the fan $\Sigma_X$.

Consider a set of homogeneous derivations $\partial_{e}$ of the polynomial algebra $\KK[x_1,\ldots,x_m]$ corresponding to some Demazure roots $e$ of the fan $\Sigma_X$.

\begin{lemma} \label{com}
Derivations $\partial_e$ and $\partial_{e'}$ commute if and only if either
$\rho_e=\rho_{e'}$ or $\langle p_e,e'\rangle=\langle p_{e'},e\rangle~=~0$.
\end{lemma}

\begin{proof}
We have $\partial_e\partial_{e'}=\partial_{e'}\partial_e$ if and only if $\partial_e\partial_{e'}(x_j)=\partial_{e'}\partial_e(x_j)$ for all $j=1,\ldots,n$. Now the lemma follows from a direct computation with formula~(*).
\end{proof}

\begin{definition} \label{deff}
A set $e_1,\ldots,e_n$ of Demazure roots of a fan $\Sigma$ of dimension $n$ is called a {\it complete collection} if $\langle p_i,e_j\rangle=-\delta_{ij}$ for all $1\le i,j\le n$.
\end{definition}

In this case, the vectors $p_1,\ldots,p_n$ form a basis of the lattice $N$, and $-e_1,\ldots,-e_n$ is the dual basis of the dual lattice $M$.

\smallskip

The following result may be considered as a combinatorial description of normalized additive actions on toric varieties.

\begin{theorem} \label{cc}
Let $X$ be a toric variety. Then additive actions on $X$ normalized by the acting torus $T$
are in bijection with complete collections of Demazure roots of the fan $\Sigma_X$.
\end{theorem}

\begin{proof}
As we have seen, a normalized additive action on $X$ determines $n$ pairwise commuting homogeneous
locally nilpotent derivations of the Cox ring $\KK[x_1,\ldots,x_m]$. They have the form
$\partial_e$ for some Demazure roots $e$.

\begin{lemma}
Homogeneous locally nilpotent derivations $\partial_1,\ldots,\partial_n$ of the Cox ring
$\KK[x_1,\ldots, x_m]$ corresponding to Demazure roots $e_1,\ldots,e_n$ define a normalized
additive action on $X$ if and only if $e_1,\ldots,e_n$ form a complete collection.
\end{lemma}

\begin{proof}
Assume first that the derivations $\partial_1,\ldots,\partial_n$ give rise to the additive actions
$\GG_a^n\times X\to X$. If some $e_i$ and $e_j$ with $i\ne j$ correspond to the same ray of the fan $\Sigma_X$, then the $\GG_a^n$-action changes at most $n-1$ coordinates of the ring $\KK[x_1,\ldots, x_m]$, and any $\GG_a^n$-orbit on $X$ can not be $n$-dimensional. Then Lemma~\ref{com} implies that $\langle p_i,e_j\rangle~=~0$ for $i\ne j$. By definition,
we have $\langle p_i,e_i\rangle~=-1$, and thus $e_1,\ldots,e_n$ form a complete collection.

Conversely, if $e_1,\ldots,e_n$ is a complete collection, then the corresponding homogeneous locally nilpotent derivations commute. They define a $\GG_a^n$-action on $\KK[x_1,\ldots, x_m]$,
and hence on $\KK^m$. This action descends to $X$. We have to show that the $\GG_a^n$-action on $X$ has an open orbit. For this purpose it suffices to check that the group $\GG_a^n\times H_X$ acts on
$\KK^m$ with an open orbit.

By construction, the group $\GG_a^n$ changes exactly $n$ of the coordinates $x_1,\ldots,x_m$, while the weights of the remaining $m-n$ coordinates with respect to the $\Cl(X)$-grading form a basis of the lattice of characters of the torus $H_X$. This shows that the stabilizer of the point $(1,\ldots,1)\in\KK^m$ in the group
$\GG_a^n\times H_X$ is trivial. Since $\dim(\GG_a^n\times H_X)=n+m-n=m$, we conclude that the $(\GG_a^n\times H_X)$-orbit of the point $(1,\ldots,1)$ is open in $\KK^m$.
\end{proof}

This completes the proof of Theorem~\ref{cc}.
\end{proof}

\begin{corollary}
A toric variety $X$ admits a normalized additive action if and only if there is a complete collection of Demazure roots of the fan $\Sigma_X$.
\end{corollary}

The following theorem shows that a normalized additive action on a toric variety is essentially unique.

\begin{theorem} \label{unique}
Any two normalized additive actions on a toric variety are isomorphic.
\end{theorem}

\begin{proof}
Define the group $\Aut(\Sigma)$ of automorphisms of a fan $\Sigma$ as the subgroup of linear transformations $\gamma$ of the lattice $N$ preserving the fan $\Sigma$. Any linear transformation $\gamma$ determines the dual transformation $\gamma^*$ of the lattice $M$. We say that two complete collections of Demazure roots of $\Sigma$ are {\it equivalent}, if one can be sent to another by an automorphism of $\Sigma$.

Let $X$ be a toric variety. Any automorphism of the fan $\Sigma_X$ is induced by an automorphism
of~$X$; see \cite[Theorem~3.3.4]{CLS}. It suffices to prove that every two complete collections
$e_1,\ldots,e_n$ and $e_1',\ldots,e_n'$ of Demazure roots of $\Sigma_X$ are equivalent; the corresponding automorphism of $X$ establishes an isomorphism of the normalized additive actions.

For the primitive vectors defined by our collections, we may assume that $p_1=p_1',\ldots,p_r=p_r'$, while $p_{r+1},\ldots,p_n,p_{r+1}',\ldots,p_n'$ are pairwise different. Then $p_{r+1}',\ldots,p_n'$ are non-positive linear combinations of $p_1,\ldots,p_n$, and
$p_{r+1},\ldots,p_n$ are non-positive linear combinations of $p_1',\ldots,p_n'$. This implies that, up to renumbering, the vector $p_j', j>r,$ has coordinate $-1$ at $p_j$ and zero coordinates at other $p_s, s>r$.

If there is a ray of the fan $\Sigma_X$ whose primitive vector $p''$ is not involved yet, then $p''$ is a non-positive linear combination of both $p_1,\ldots,p_n$ and $p_1',\ldots,p_n'$. This shows that $p''$ is a non-positive linear combination of $p_1,\ldots,p_r$.

Consider the automorphism $\gamma$ of the lattice $N$ which sends the basis $p_1,\ldots,p_n$ to the basis $p_1',\ldots,p_n'$. It is easy to see that $\gamma$ sends $p_1',\ldots,p_n'$ to $p_1,\ldots,p_n$, and preserves all the vectors $p''$. Thus $\gamma$ sends one complete collection of Demazure roots to another.

It remains to check that $\gamma$ is an automorphism of the cone $\Sigma_X$. Consider the vector $p_j$, $j>r$, and the corresponding Demazure root $e_j$. Then all vectors $p_i$, $i\ne j$,
$p'_i$, $i\ne j$, and $p''$ are contained in the hyperplane $\langle\cdot,e_j\rangle=0$. Thus we can decompose the transformation~$\gamma$ in the composition of transformations $\gamma_{r+1},\ldots,\gamma_n$, where $\gamma_j$ sends $p_j$ to $p_j'$ and fixes the hyperplane $\langle\cdot,e_j\rangle=0$ pointwise.

It suffices to check that each $\gamma_j$ is an automorphism of the fan $\Sigma_X$. Equivalently, we have to check that for any cone $\sigma$ in $\Sigma_X$ containing $p_j$ there is a cone in $\Sigma_X$ generated by the same rays, but with $\rho_j$ replaces by $\rho_j'$. If the cone $\sigma$ contains both $p_j$ and $p_j'$, the assertion is clear. If $\sigma$ does not contain $p_j'$, it is generated by $p_j$ and some cone $\sigma_0$ contained in the hyperplane $\langle\cdot,e_j\rangle=0$. By the symmetric arguments, the cone $\sigma_0$ is contained in the hyperplane $\langle\cdot,e'_j\rangle=0$ as well. Condition~(2) from the definition of a Demazure root of a fan $\Sigma$ implies that the cone generated by $p'_j$ and $\sigma_0$ belongs to $\Sigma$ as well. This completes the proof of Theorem~\ref{unique}.
\end{proof}


\section{Arbitrary additive actions}
\label{sec4}

Let $X$ be a complete toric variety with an acting torus $T$. It is well known that the automorphism group $\Aut(X)$ is a linear algebraic group with $T$
as a maximal torus, see~\cite{De}, \cite{Cox}. In particular, $\Aut(X)$ contains a maximal unipotent subgroup $U$, and all such subgroups are conjugate
in $\Aut(X)$.

Let $G$ be the product of the semisimple part and the unipotent radical of $\Aut(X)$.

\begin{theorem} \label{3con}
Let $X$ be a complete toric variety with an acting torus $T$. The following conditions are equivalent.
\begin{itemize}
\item[(1)]
There exists an additive action on $X$ normalized by the acting torus $T$.
\item[(2)]
There exists an additive action on $X$.
\item[(3)]
A maximal unipotent subgroup $U$ of the automorphism group $\Aut(X)$ acts on $X$ with an open orbit.
\item[(4)]
The subgroup $G$ acts on $X$ with an open orbit.
\end{itemize}
\end{theorem}

\begin{proof}
Clearly, condition (1) implies (2). Since any unipotent subgroup of $\Aut(X)$ is contained in a maximal unipotent subgroup, implication (2) $\Rightarrow$ (3) holds.

Now we prove implication (3) $\Rightarrow$ (1). Let a maximal unipotent subgroup $U$ act on $X$ with an open orbit. We may assume that $U$ is normalized by $T$. Then $U$ is generated by one-parameter subgroups $U_e$ normalized by $T$ and corresponding to some Demazure roots $e$.

Let $\mathcal{O}$ be an open $U$-orbit in $X$. Being an orbit of a unipotent group, $\mathcal{O}$ is isomorphic to the affine space $\KK^n$. Since $U$ is normalized by $T$, the subset $\mathcal{O}$
is a $T$-invariant open affine chart on $X$. It implies that there are coordinates $x_1,\ldots,x_n$ on $\mathcal{O}$ such that the $T$-action on $\mathcal{O}$ is the standard action of $T$ on $\KK^n$.

The root subgroups $U_e$ acts on $\KK^n$ as
$$
x_i \to x_i + sx_1^{c_1}\ldots x_n^{c_n}, \ s\in\KK, \quad x_j\to x_j \quad \forall j\ne i,
$$
for some $i$, where the monomial $x_1^{c_1}\ldots x_n^{c_n}$ does not depend on $x_i$. It corresponds to a locally nilpotent derivation $\partial_e$ given by
$$
\partial_e(x_i)=x_1^{c_1}\ldots\widehat{x_i}\ldots x_n^{c_n}, \quad \partial_e(x_j)=0 \ \forall j\ne i,
$$
where $e=(c_1,\ldots,c_n)$ is a Demazure root with $c_i=-1$.

Let $U_i$ be a subgroup of the automorphism group of $\KK^n$ sending $x_i$ to $x_i+c$, $c\in\KK$, and fixing all $x_j$, $j\ne i$. We claim that the group $U$ contains all the subgroups $U_i$, $i=1,\ldots,n$.

Let us assume the converse. By renumbering we may suppose that for some $r<n$ precisely the subgroups $U_j$ with $j\le r$ are contained in $U$. By $r=0$ we mean that there are no subgroup $U_i$ in $U$ at all.

Since the action of $U$ on $\KK^n$ is transitive, for every $i>r$ there is at least one subgroup $U_e$ of the form $x_i \to x_i + sx_1^{c_1}\ldots x_n^{c_n}$ in $U$. If for all such subgroups the monomial $x_1^{c_1}\ldots x_n^{c_n}$ depends at least on one variable $x_{r+1},\ldots,x_n$, the subspace $x_{r+1}=\ldots=x_n=0$ in $\KK^n$ is invariant under $U$, a contradiction. Thus we have a subgroup $U_e$ of the form $x_i\to x_i+sx_1^{c_1}\ldots x_r^{c_r}$ for some $i>r$.

Consider the subspace spanned by the derivations $\partial_e$ over all subgroups $U_e$ contained in $U$ in the Lie algebra of all derivations of the polynomial algebra $\KK[x_1,\ldots,x_n]$. This subspace may be identified with the tangent algebra of the group $U$. In particular, it is closed under the Lie bracket of derivations.

Let $\partial_j$ with $\partial_j(x_i)=\delta_{ij}$ be the derivation corresponding to $U_j$ and $\partial_e$ correspond to the subgroup $U_e$ of the form $x_i\to x_i+sx_1^{c_1}\ldots x_r^{c_r}$ for some $i>r$.

It is easy to check that the commutant $[\partial_1, \partial_e]$ is a locally nilpotent derivation sending $x_i$ to $c_1x_1^{c_1-1}\ldots x_r^{c_r}$ and annihilating all other variables. Taking more commutants of $\partial_e$ with $\partial_1,\ldots,\partial_r$, we obtain a locally nilpotent derivation, which sends $x_i$ to a nonzero constant and annihilates all other variables. It shows that the derivation $\partial_i$ is contained in the Lie algebra of $U$. Thus the subgroup $U_i$ is contained in $U$, a contradiction with a choice of $r$.

We conclude that the subgroups $U_i$, $i=1,\ldots,n$, are contained in $\Aut(X)$ and generate a commutative unipotent subgroup normalized by $T$, which acts on $\mathcal{O}$ transitively. Thus implication (3) $\Rightarrow$ (1) is proved.

Since any maximal unipotent subgroup is contained in $G$, we have (3) $\Rightarrow$ (4).

Let us prove implication (4) $\Rightarrow$ (3). Since $X$ is irreducible, the open $G$-orbit on $X$ intersects the open $T$-orbit $X_0$. Take a point $x_0$ in $X_0$ whose $G$-orbit is open in $X$.

Note that $G$ coincides with the subgroup of $\Aut(X)$ generated by all root subgroups $U_e$.
By Proposition~\ref{phe}, the intersection of the orbit $U_e\cdot x_0$ with $X_0$  coincides with the orbit $R_e\cdot x_0$, where $R_e$ is a one-parameter subtorus in $T$ represented by the distinguished vector $p_e$ of the Demazure root $e$. Thus the orbit $G\cdot x_0$ is open in $X$ if and only if the subtori $R_e$ are not contained in a proper subtorus of $T$ or, equivalently, the vectors $p_e$'s are not contained in a proper subspace of $N_{\mathbb{Q}}$.

Hence we have a set $p_1,\ldots,p_n$ of linearly independent primitive vectors of the fan $\Sigma_X$ corresponding to some Demazure roots $e_1,\ldots,e_n$. Let $U$ be a maximal unipotent subgroup of $\Aut(X)$ normalized by $T$. If the root subgroup $U_{e_i}$ is contained in the unipotent radical of $\Aut(X)$, then $U_{e_i}$ is contained in $U$ as well. If $U_{e_i}$ is in the semisimple part of $\Aut(X)$ and is not in $U$, then the opposite root subgroup $U_{-e_i}$ is in $U$. The hyperplane $\langle\cdot,e_i\rangle=0$ contains all the rays of the fan $\Sigma_X$ except for $\rho_{e_i}$ and $\rho_{-e_i}$. Thus replacing the vector $p_i$ by $p_{-e_i}$ in the set $p_1,\ldots,p_n$, we again obtain a linearly independent set.

Proceeding this way, we obtain a linearly independent set $p_1,\ldots,p_n$, where the vectors $p_i$ correspond to some root subgroups in $U$. Thus $U$ acts on $X$ with an open orbit, and the proof of implication (4) $\Rightarrow$ (3) is completed.
\end{proof}

\begin{corollary}\label{cba}
A toric variety $X$ admits an additive action if and only if there is a complete collection of Demazure roots of the fan $\Sigma_X$.
\end{corollary}

\begin{remark}
As we have seen in the proof, condition~(4) of Theorem~\ref{3con} is equivalent to existence of linearly independent primitive vectors $p_1,\ldots,p_n$ in the fan $\Sigma_X$ corresponding to some Demazure roots $e_1,\ldots,e_n$. At the same time we know that condition~(1) of Theorem~\ref{3con} means that there is a complete collection of Demazure roots of the fan $\Sigma_X$. It seems to be an interesting problem to prove explicitly that linearly independent vectors $p_1,\ldots,p_n$ correspond to a complete collection of Demazure roots.
\end{remark}

Let $\fu$ be the Lie algebra of a maximal unipotent subgroup $U$ of $\Aut(X)$ normalized by $T$. Consider the Grassmannian $\Gr(n,\fu)$ of $n$-dimensional subspaces of the space $\fu$.
Commutative Lie subalgebras of dimension $n$ form a closed subvariety $Z$ in $\Gr(n,\fu)$. The adjoint action of the torus $T$ on $\fu$ induces an action of $T$ on $\Gr(n,\fu)$, which leaves invariant the subvariety $Z$. Points in $Z$ fixed by $T$ are precisely $T$-normalized commutative subalgebras in~$\fu$.

Take any commutative $n$-dimensional subgroup $H$ in $U$ which acts on $X$ with an open orbit, and consider the closure of the $T$-orbit of the point in $Z$ corresponding to the Lie algebra of~$H$. This closure contains finitely many $T$-fixed points. If $H$ is not normalized by $T$, then there are at least two fixed points in the closure. In general, not any subgroup corresponding to such points acts on $X$ with an open orbit, see Examples~\ref{p2} and~\ref{hir} below. We expect that
at least one such subgroup does act with an open orbit on $X$.


\section{Projective toric varieties and polytopes}
\label{sec5}

It is well known that there is a correspondence between convex lattice polytopes and projective toric varieties. The aim of this section is to characterize polytopes corresponding to toric varieties that admit an additive action.

We begin with preliminary results; see \cite[Chapter~2]{CLS} and \cite[Section~1.5]{Fu} for more details. Let $M$ be a lattice of rank $n$ and $P$ be a full dimensional convex polytope in the space $M_{\QQ}$. We say that $P$ is a {\it lattice polytope} if all its vertices lie in $M$.

A subsemigroup $S\subseteq M$ is called {\it saturated} if $S$ coincides with the intersection of
the group $\ZZ S$ and the cone $\QQ_{\ge 0}S$ it generates. A lattice polytope $P$ is {\it very ample} if for every vertex $v\in P$, the semigroup $S_{P,v}:=\ZZ_{\ge 0}(P\cap M - v)$ is saturated. It is known that for every lattice polytope $P$ and every $k\ge n-1$ the polytope $kP$ is very ample, see \cite[Corollary~2.2.19]{CLS}.

Let us consider $M$ as a lattice of characters of a torus $T$. Let $P\subseteq M_{\QQ}$ be a very ample polytope and $P\cap M=\{m_1,\ldots,m_s\}$. We consider a map
$$
T\to\PP^{s-1}, \quad t\mapsto (\chi^{m_1}(t):\ldots:\chi^{m_s}(t))
$$
and define a variety $X_P$ as the closure of the image of this map in $\PP^{s-1}$.
It is known that $X_P$ is a projective toric variety with the acting torus $T$, and any projective toric variety appears this way.

\begin{definition} \label{iir}
We say that a very ample polytope $P$ is {\it inscribed in a rectangle} if there is a vertex $v_0\in P$ such that
\begin{enumerate}
\item[(1)]
the primitive vectors on the edges of $P$ containing $v_0$ form a basis $e_1,\ldots,e_n$
of the lattice $M$;
\item[(2)]
for every inequality $\langle p,x\rangle\le a$ on $P$ that corresponds to a facet of $P$ not passing through $v_0$ we have $\langle p,e_i\rangle\ge 0$ for all $i=1,\ldots,n$.
\end{enumerate}
\end{definition}

\begin{figure}[h]\label{rectangle-polytope}
\centerline{\includegraphics[scale = 0.12]{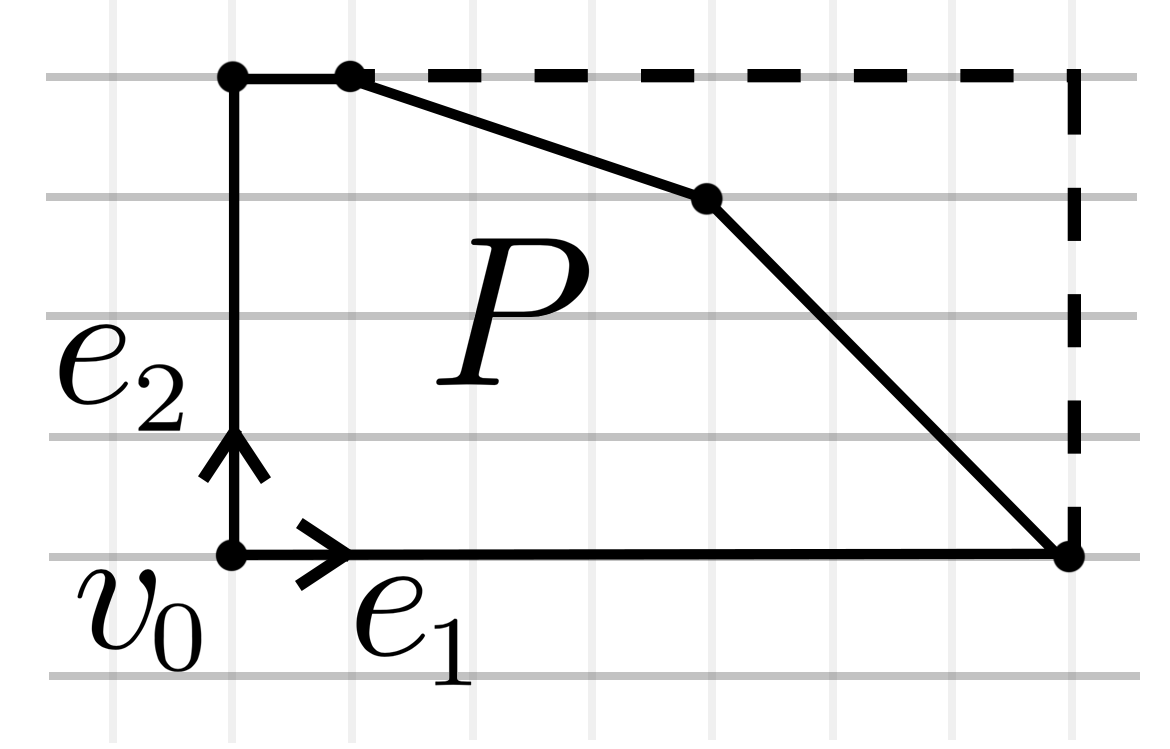}}
\centerline{\figurename{\ 1.} A polytope $P$ inscribed in a rectangle}
\end{figure}

\medskip

The following result was communicated to us by Evgeny Feigin as a conjecture.

\begin{theorem}
Let $P$ be a very ample polytope and $X_P$ be the corresponding projective toric variety. Then $X_P$ admits an additive action if and only if the polytope $P$ is inscribed in a rectangle.
\end{theorem}

\begin{proof}
By Corollary~\ref{cba}, a toric variety $X$ admits an additive action if and only if the fan $\Sigma_X$ admits a complete collection of Demazure roots. By \cite[Proposition~3.1.6]{CLS}, the fan $\Sigma_{X_P}$ of the toric variety $X_P$ corresponding to the polytope $P$ coincides with the normal fan $\Sigma_P$ of the polytope~$P$. It is straightforward to check that two conditions of Definition~\ref{iir} are precisely the conditions on $-e_1,\ldots,-e_n$ to be a complete collection of Demazure roots of the fan $\Sigma_P$.
\end{proof}

\begin{example}
The segment $P=[0,d]$ in $\QQ^1$ with $d\in\ZZ_{\ge 1}$ is a polytope inscribed in a rectangle, and the variety
$$
X_P=\overline{\{(1:a:\ldots:a^d) \, ; \, a\in\KK\}}\subseteq\PP^d
$$
is a rational normal curve of degree $d$.
\end{example}

\begin{example}
The polytope

\begin{figure}[h]\label{hirtzebruch-polytope}
\centerline{\includegraphics[scale = 0.25]{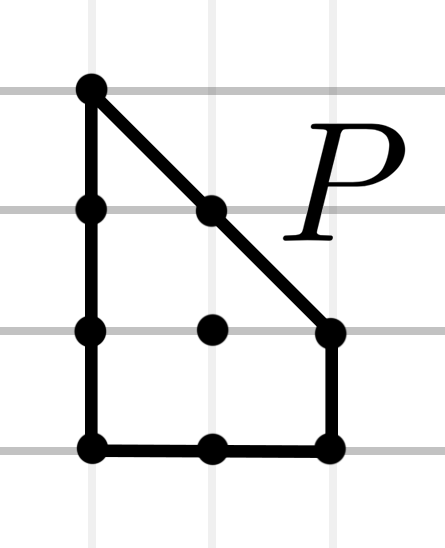}}
\centerline{\figurename{\ 2.} The polytope of a Hirzebruch surface}
\end{figure}

\noindent defines the surface
$$
X_P=\overline{\{(1:a:a^2:b:ab:a^2b:b^2:ab^2:b^3) \, ; \, a,b\in\KK\}}\subseteq\PP^8
$$
isomorphic to the Hirzebruch surface.
\end{example}


\section{Examples and open problems}
\label{sec6}

We begin this section with explicit formulas for additive actions on toric varieties in terms of Cox rings.

\begin{example}
The normalized additive action on the projective space $\PP^n$ is given by
$$
(x_1,\ldots,x_n,x_{n+1})\mapsto (x_1+s_1x_{n+1},\ldots,x_n+s_nx_{n+1},x_{n+1}).
$$
The hyperplane $x_{n+1}=0$ consists of $\GG_a^n$-fixed points and thus for $n\ge 2$ the number of $\GG_a^n$-orbits on $\PP^n$ is infinite.
\end{example}

\begin{example}
The normalized additive action on the product $\PP^1\times\ldots\times\PP^1$ is given by
$$
(x_1,x_2\ldots,x_{2n-1},x_{2n})\mapsto (x_1+s_1x_2,x_2,\ldots,x_{2n-1}+s_nx_{2n},x_{2n}).
$$
This shows that the number of $\GG_a^n$-orbits on $\PP^1\times\ldots\times\PP^1$ is $2^n$.
\end{example}

\begin{example} \label{p2}
By~\cite[Proposition~3.2]{HT}, every additive action on $\PP^2$ has the form
$$
(x_1,x_2,x_3)\mapsto (x_1+as_1x_2+s_2x_3, x_2+bs_1x_3,x_3)
$$
with some fixed $a,b\in\KK$. For $a\ne 0$ and $b\ne 0$ we have one isomorphy class of non-normalized additive actions. Every such action has three orbits on $\PP^2$. With $a=0$ it degenerates to a normalized additive action, while with $b=0$ it degenerates to a $\GG_a^2$-action on $\PP^2$ with generic one-dimensional orbits, which is not an additive action.
\end{example}

\begin{example} \label{hir}
Let $X$ be the Hirzebruch surface $\FF_d$. Its fan is generated by the vectors
$$
p_1=(1,0), \quad p_2=(0,1), \quad p_3=(-1,d), \quad p_4=(0,-1)
$$
with some $d\in\ZZ_{\ge 1}$. The Cox ring $\KK[x_1,x_2,x_3,x_4]$ carries a $\ZZ^2$-grading given by $$
\dg(x_1)=(1,0), \quad \dg(x_2)=(0,1), \quad \dg(x_3)=(1,0), \quad \dg(x_4)=(d,1).
$$
Moreover, $X$ is obtained as a geometric quotient of $\widehat{X}=\KK^4\setminus(\{x_1=x_3=0\}\cup\{x_2=x_4=0\})$ by the action of the torus $H_X=(\KK^{\times})^2$, whose action is given by the $\ZZ^2$-grading.

In this case the Demazure roots are $(1,0)$, $(-1,0)$ and $(k,1)$ with $0\le k\le d$, and the corresponding homogeneous locally nilpotent derivations are given as
$$
x_1\frac{\partial}{\partial x_3}, \quad x_3\frac{\partial}{\partial x_1}, \quad
x_1^kx_2x_3^{d-k}\frac{\partial}{\partial x_4}.
$$
There are two complete collections of Demazure roots, namely $(1,0), (d,1)$ and $(-1,0),(0,1)$.
They define two normalized additive actions
$$
(x_1,x_2,x_3,x_4)\mapsto (x_1,x_2,x_3+s_1x_1,x_4+s_2x_1^dx_2)
$$
and
$$
(x_1,x_2,x_3,x_4)\mapsto (x_1+s_1x_3,x_2,x_3,x_4+s_2x_2x_3^d),
$$
which are rearranged by the automorphism $(x_1,x_2,x_3,x_4)\mapsto (x_3,x_2,x_1,x_4)$.

\smallskip

Applying the results of~\cite{De} or~\cite{Cox} we obtain
$$
\Aut(X)\cong\KK^{\times}\cdot\PSL(2)\rightthreetimes\,\GG_a^{d+1}.
$$

For $d=1$ a maximal unipotent subgroup of $\Aut(X)$ acts as
$$
(x_1,x_2,x_3,x_4)\mapsto (x_1+ax_3,x_2,x_3,x_4+bx_1x_2+cx_2x_3).
$$
Every point $[a:b]\in\PP^2$ defines a $\GG_a^2$-action on $\FF_1$, which is an additive action if
$[a:b]\ne[0:1]$. With $[a:b]=[1:0]$ we have a normalized additive action, while all other points
$[a:b]$ define pairwise isomorphic non-normalized additive actions. Thus there are exactly two isomorphy classes of additive actions on $\FF_1$. This result is obtained in~\cite[Proposition~5.5]{HT} by geometric arguments.
\end{example}

In~\cite[Theorem~6.2]{HT}, a classification of additive actions on smooth projective threefolds with Picard group of rank~1 is given. The following result may be considered as a generalization of this classification to the case of not necessary smooth varieties of arbitrary dimension in toric setting.

\begin{proposition}
Let $X$ be a complete toric variety with $\rk\Cl(X)=1$. Then $X$ admits an additive action if and only if $X$ is a weighted projective space $\PP(1,d_1,\ldots,d_n)$ for some positive integers $d_i$.
\end{proposition}

\begin{proof}
First assume that $X$ admits a (normalized) additive action. We know that the group $\Cl(X)$ is free and thus $\Cl(X)\cong\ZZ$. This allows to assume that the primitive vectors on the rays of the fan $\Sigma_X$ are
$$
(1,0,\ldots,0),\ldots,(0,0,\ldots,1),(-a_1,\ldots,-a_n), \ \ a_i\ge 0.
$$
Since the fan $\Sigma_X$ is complete we have $a_i>0$. Putting $d_i=a_i$ we obtain the fan of the weighted projective space~$\PP(1,d_1,\ldots,d_n)$, see~\cite[Section~2.2]{Fu}.

Conversely, such a fan possesses a complete collection of Demazure roots, and thus $\PP(1,d_1,\ldots,d_n)$ admits an additive action.
\end{proof}

\begin{remark}
In terms of Cox rings, a normalized additive action on $\PP(1,d_1,\ldots,d_n)$ is given by
$$
(x_1,\ldots,x_n,x_{n+1})\mapsto (x_1+s_1x_{n+1}^{d_1},\ldots,x_n+s_nx_{n+1}^{d_n},x_{n+1})
$$
and thus the complement of the open $\GG_a^n$-orbit on $\PP(1,d_1,\ldots,d_n)$ consists of
$\GG_a^n$-fixed points.
\end{remark}

Let us finish this paper with several questions and open problems.

\begin{problem}
Find all toric varieties which admit a normalized additive action with finitely many orbits.
\end{problem}

It is shown in \cite[Section~3]{HT} that the projective space $\PP^n$ with $n\ge 6$ admits infinitely many non-isomorphic additive actions. This rises the following problem.

\begin{problem}
Describe all additive actions on a given toric variety $X$. When the number of isomorphy classes of such actions on $X$ is infinite?
\end{problem}

As we have shown in this paper, if a toric variety $X$ of dimension $n$ admits an additive action,
then the Cox ring $\KK[x_1,\ldots,x_m]$ is $\ZZ^{m-n}$-graded, the degrees of $x_1,\ldots,x_{m-n}$ form a basis of the lattice $\ZZ^{m-n}$, and $\deg x_{m-n+1},\ldots,\deg x_m$ are non-positive combinations of these basis vectors. Moreover, if $X$ is either complete or affine, then these conditions are necessary and sufficient. These observations motivate the following problem.

\begin{problem}
Characterize algebraic varieties admitting an additive action in terms of their Cox rings.
\end{problem}


\end{document}